\documentclass[12pt,a4paper,twoside,reqno]{amsart}
\usepackage[utf8]{inputenc}
\usepackage[margin=0.9in]{geometry} 
\usepackage{amsthm}
\usepackage{amsmath}
\usepackage{amsfonts}
\usepackage{amssymb}
\usepackage{chngcntr}
\usepackage{url}
\usepackage{hyperref}
\usepackage[makeroom]{cancel}
\usepackage{enumerate}
\usepackage{mathrsfs}
\newtheorem{Theorem}{Theorem}
\newtheorem{Conjecture}[Theorem]{Conjecture}
\newtheorem{Proposition}[Theorem]{Proposition}

\newtheorem{Lemma}[Theorem]{Lemma}

\newtheorem{Definition}[Theorem]{Definition}
\newtheorem{Remark}[Theorem]{Remark}

\counterwithin{Theorem}{section}
\numberwithin{equation}{section}
\title{Correlations of sieve weights and distributions of zeros}
\author[Aled Walker]{Aled Walker}
\address{Trinity College, Cambridge, UK, CB2 1TQ}
\email{aw530@cam.ac.uk}
\date{}
\begin{document}
\begin{abstract}
In this note we give two small results concerning the correlations of the Selberg sieve weights. We then use these estimates to derive a new (conditional) lower bound on the variance of the primes in short intervals, and also on the so-called `form factor' for the pair correlations of the zeros of the Riemann zeta function. Our bounds ultimately rely on the estimates of Bettin--Chandee for trilinear Kloosterman fractions.
\end{abstract}
\maketitle

\section{Introduction}
For $z \geqslant 1$, we define the weight 
\begin{equation}
\label{equation rho}
\rho_{z,d}: = \frac{d \mu(d)}{\varphi(d)} \sum\limits_{\substack{ q \leqslant z/d \\ (q,d) = 1}} \frac{\mu^2(q)}{\varphi(q)},
\end{equation}
\noindent writing
\begin{equation}
\label{equation defining lambdaQ}
\lambda_z(n) := \sum\limits_{d \vert n} \rho_{z,d}.
\end{equation}
\noindent These weights arise in the theory of the Selberg sieve. A classical estimate (proved in \cite[Lemma 2]{Go95}, for instance) reads 
\begin{equation}
\label{equation classical estimate}
\sum\limits_{n \leqslant X} \lambda_z(n)^2 = X L(z) + O(z^2),
\end{equation}
\noindent where
\begin{equation}
\label{eq: L of z}
L(z): = \sum\limits_{ q \leqslant z} \frac{\mu^2(q)}{\varphi(q)}.
\end{equation}
\noindent Since $\lambda_z(n) = L(z)$ if all of the prime factors of $n$ are greater than $z$, expression (\ref{equation classical estimate}) may be used as an upper-bound sieve for an interval.

However, sieving is not our present concern. We are motivated instead by a different strand of the literature, in which $\lambda_z(n)$ acts as a proxy for the von Mangoldt function $\Lambda(n)$ and information concerning difficult properties of $\Lambda(n)$ may be gleaned from studying the more tractable properties of $\lambda_z(n)$. Elements of this approach date back to Heath-Brown and his work on the ternary Goldbach problem \cite{HB85}, and it also appears in work of Goldston \cite{Go95} on the variance of primes in short intervals, Friedlander--Goldston \cite{FG96} on the distribution of primes in arithmetic progressions, Goldston--Yildrim \cite{GY01} on the variance of primes in short arithmetic progressions, as well as in the work of Goldston--Gonek--\"{O}zl\"{u}k--Snyder \cite{GGOS00} on the distribution of the zeros of the Riemann zeta function (which we will address at length below). 

In certain contexts the functions derived from $\lambda_z(n)$ have been called `Vaughan's approximation', after Vaughan's papers \cite{Va03, Va03a} which considered in detail the quality of the approximation \[\Big\vert \sum\limits_{\substack{n \leqslant x \\ n \equiv a \, (\text{mod } q)}} \lambda_z(n) - \psi(x;q,a)\Big\vert\] for small $z$. This terminology is used in Fiorilli's work \cite{Fi17}, say.

We will be interested in two estimates for $\lambda_z(n)$: the estimate (\ref{equation classical estimate}) and the off-diagonal correlations $\lambda_z(n) \lambda_z(n+k)$. An old result for these off-diagonal correlations is as follows: 
\begin{Proposition}
\label{Prop: classical off diagonal}
Let $1 \leqslant z \leqslant X$, and let $k$ be an integer such that $1 \leqslant \vert k\vert \leqslant X$. Let $X_1 = \max(0,-k)$ and $X_2 = \min(X,X-k)$. Then 
\begin{equation}
\label{eq: off diagonal classical}
\sum\limits_{ X_1 < n \leqslant X_2} \lambda_z(n) \lambda_z(n+k) = \mathfrak{S}(k)(X- \vert k\vert ) + O\Big( \frac{ k \tau(k) X}{ \varphi(k) z}\Big) + O(z^2),
\end{equation} where 
\begin{equation}
\label{eq: singular series}
\mathfrak{S}(k) = \begin{cases} 
2 \prod\limits_{p>2} \Big(1 - \frac{1}{(p-1)^2}\Big)\cdot \prod\limits_{\substack{p > 2 \\ p \vert k}} \Big( \frac{p-1}{p-2}\Big)  & k \text{ even;} \\
0 & k \text{ odd}. \end{cases} 
\end{equation}
\end{Proposition} 
\begin{proof}
See \cite[Lemma 2]{Go95}. 
\end{proof}
\noindent Estimates (\ref{equation classical estimate}) and (\ref{eq: off diagonal classical}) are non-trivial only in the range $\log X/\log z > 2$, which is the familiar `square-root barrier' for the upper-bound sieve. Our main technical results in this note extend the range of applicable $z$ beyond $X^{1/2}$. 

For the off-diagonal case, we can proceed unconditionally.
\begin{Proposition} 
\label{Proposition lambdaQ correlation}
With notation as in Proposition \ref{Prop: classical off diagonal}, for any $\varepsilon >0$ we have
\[\sum\limits_{X_1 < n \leqslant X_2} \lambda_z(n) \lambda_z(n+k) = \mathfrak{S}(k) (X - \vert k\vert) + O\Big(\frac{k \tau(k)X}{ \varphi(k)z}\Big) + O_\varepsilon(\min(z^2,X^{\frac{47}{74} + \varepsilon}z^{\frac{53}{74}})).\]
\end{Proposition}
\noindent For fixed $k$ this bound is non-trivial if $X^{\varepsilon} \leqslant z \leqslant X^{\frac{27}{53} - \varepsilon}$. (Here we follow the convention that $\varepsilon$ refers to a small parameter that may change from line to line.)

Our proof of Proposition \ref{Proposition lambdaQ correlation} is short, and ultimately relies on the deep results of Bettin--Chandee on trilinear Kloosterman fractions \cite{BC18} (which are themselves based on the pioneering work of Duke--Friedlander--Iwaniec \cite{DFI97} on bilinear Kloosterman fractions). The result is new, but it came to our attention while preparing the final version of our manuscript that something similar does appear in earlier work of Coppola--Laporta \cite[Lemma 3]{CL16}. Their work only used the older DFI bound, and therefore only produced a non-trivial error term in the weaker range $z \leqslant X^{48/95}$; furthermore their way of proceeding from the Kloosterman fraction bounds to correlation bounds was more complicated (as they used sharp cut-offs, as opposed to smooth cut-offs). However, the core ideas seem to be broadly similar. 

The work of Bettin--Chandee was also used by Fouvry--Radziwi\l\l\ \cite{FR18, FR18a}, who have established very strong equidistribution results for bilinear sums in arithmetic progressions. In turn, they used these bounds (in \cite[Corollary 1.5]{FR18}) to establish equidistribution of sieve weights over certain arithmetic progressions with large modulus, where the level of distribution of the sieve weight can be as large as $X^{1/2 + \delta}$ for some $\delta >0$. Their Corollary 1.5 cannot be directly applied for our purposes (the range of permissible arithmetic progressions is too restrictive and they only save a power of $\log X$), but it seems likely that their methods could be adapted to yield a version of Proposition \ref{Proposition lambdaQ correlation}. On balance, then, we feel that our main contribution here is the direct manner of proof rather than the statement itself.

For the diagonal case, we are forced to proceed conditionally to obtain a power saving. 
\begin{Proposition}[Method of Hooley, \cite{Ho00}]
\label{Proposition lambdaQ squared on RH}
Assume the Riemann Hypothesis. Then, for all $1 \leqslant z \leqslant X$ and for all $\varepsilon>0$,
\[\sum\limits_{n \leqslant X} \lambda_z^2(n) = XL(z) + O_{\varepsilon}(\min(z^2,X^{1/2 + \varepsilon} z^{1/2})).\]
\end{Proposition}
\noindent This bound is non-trivial if $z \leqslant X^{1-\varepsilon}$. As we will explain in Section \ref{Sec:Hooley method}, the proof of Proposition \ref{Proposition lambdaQ squared on RH} is a straightforward adaptation of a method of Hooley, who studied the weight $\lambda_z(n)$ as part of his long series of papers on the Barban--Davenport--Halberstam theorem. Assuming RH is of course regrettable, but, for the applications we have in mind, this assumption is standard. 

The proofs of Propositions \ref{Proposition lambdaQ correlation} and \ref{Proposition lambdaQ squared on RH} will be deferred to Sections \ref{Sec:offdiagonal} and \ref{Sec:Hooley method} respectively. For the rest of this introduction, we will discuss our two applications.
\\

Firstly, we wish to study the pair correlations of the zeros of the Riemann zeta function. In particular we wish to understand the behaviour of the function $F_T(\alpha)$, which was first introduced by Montgomery in his seminal paper \cite{Mo73}.

\begin{Definition}[Two point form factor\footnote{This is terminology from random matrix theory, which we learnt from Lagarias--Rogers \cite{LR19}.}]
\label{Definition Montgomery form factor}
For $\alpha \in \mathbb{R}$ and $T\geqslant 2$ we define
\[ F_T(\alpha): = \Big(\frac{T \log T}{2\pi} \Big)^{-1} \sum\limits_{ 0 < \gamma, \gamma^\prime \leqslant T} e\Big(\alpha \frac{\log T}{2\pi}(\gamma - \gamma^\prime)\Big) w( \gamma - \gamma^\prime),\] where $w(u) = 4/(4+ u^2)$, and where $\gamma,\gamma^\prime$ range over the imaginary parts of the zeros of the Riemann zeta function $\zeta(s)$ (counted with multiplicity).  Here, as throughout, we let $e(\alpha)$ be a shorthand notation for $e^{2\pi i \alpha}$. 
\end{Definition}
\noindent  Apart from the weight function $w$, which might seem a little odd at first sight, the function $F_T(\alpha)$ is evidently a natural object to consider. Indeed, it is an old fact, going back to Riemann, that the number of zeros with imaginary part between $0$ and $T$ is asymptotically equal to $\frac{1}{2\pi} T\log T$. The average gap between consecutive zeros is therefore $2 \pi / \log T$, and so the function $F_T(\alpha)$ measures the (weighted) Fourier transform of the gaps between the zeros, measured at the scale of the average gap. 

We observe the trivial facts that $F_T(\alpha)= F_T(-\alpha) = \overline{F_T(\alpha)}$, so from now on we only consider $\alpha \geqslant 0$. In \cite{Mo73}, Montgomery proved the following asymptotic result for $F_T(\alpha)$: 

\begin{Theorem}[Theorem 1, \cite{Mo73}]
\label{Theorem Montgomery first theorem}
Assume the Riemann Hypothesis. Let $\varepsilon >0$ and let $0 \leqslant \alpha \leqslant 1 - \varepsilon$. Then \[ F_T(\alpha) = T^{-2  \alpha} (\log T)(1 + o_\varepsilon(1)) + \alpha  + o_\varepsilon(1)\] as $T \rightarrow \infty$. 
\end{Theorem}
\noindent In particular, for a fixed $\alpha \in (0,1)$ one has \[ F_T(\alpha) = \alpha + o(1)\] as $T \rightarrow \infty$.

Montgomery's proof of Theorem \ref{Theorem Montgomery first theorem} used a version of the explicit formula for $\zeta(s)$, in order to link $F_T(\alpha)$ to a variance estimate for the Dirichlet polynomial $\sum_{n \leqslant T^\alpha} \Lambda(n) n^s$. Based on this method, together with some conjectures on the size of the error term in the twin prime conjecture, Montgomery was moved to posit the following:

\begin{Conjecture}[Montgomery's pair correlation conjecture]
\label{Conjecture Montgomery pair corr}
For each fixed $\alpha \in [1,\infty)$ one has \[ F_T(\alpha) = 1 + o(1)\] as $T \rightarrow \infty$. 
\end{Conjecture}
\noindent There are now substantial computations of zeros which suggest that Conjecture \ref{Conjecture Montgomery pair corr} is true (see \cite{Od87}). 

As we will discuss later in this introduction, Conjecture \ref{Conjecture Montgomery pair corr} is tied up with deep issues concerning the variance of the number of primes in short intervals \cite{GM87}.  Nonetheless, there are some intriguing partial results, particularly the following lower-bound of Goldston--Gonek--\"{O}zl\"{u}k--Snyder. 

\begin{Theorem}[G--G--\"{O}--S,\cite{GGOS00}]
\label{Theorem GGOS}
Assume the Generalised Riemann Hypothesis for Dirichlet $L$-functions. Then for any $\varepsilon >0$ one has 
\begin{equation}
\label{equation GGOS}
F_T(\alpha) \geqslant \frac{3}{2} -  \alpha - \varepsilon,
\end{equation}
\noindent provided $1 \leqslant \alpha  \leqslant \frac{3}{2} - \varepsilon$ and $T \geqslant T_0(\varepsilon)$. 
\end{Theorem}
\noindent In particular, for fixed $\alpha \in [1,3/2)$ we have 
\begin{equation}
\label{eq: AH FT}
\liminf \limits_{T \rightarrow \infty} F_T(\alpha) \geqslant \frac{3}{2} - \alpha.
\end{equation} This bound is an improvement over the trivial bound $F_T(\alpha) \geqslant 0$, which follows from the formula \[F_T(\alpha) = \frac{2}{\pi} \int\limits_{-\infty}^{\infty} \Big\vert \sum\limits_{0 < \gamma \leqslant T} \frac{e(\alpha \gamma \frac{\log T}{2 \pi})}{1 + (t - \gamma)^2}\Big\vert^2 \, dt\] proved by Montgomery (see \cite[Equation (4.8)]{GGOS00}).

The bound \eqref{eq: AH FT} can be compared with the bound under the so-called `Alternative Hypothesis' (AH). This is a pathological distribution for the imaginary parts of the zeros of $\zeta(s)$, which, among other things, would imply that \[ \lim\limits_{T \rightarrow \infty} F_T(\alpha) = 2 - \alpha \] for $\alpha \in (1,2).$ The relationship between the bounds in this paper and AH will be discussed in Section \ref{Sec: AH}. 

We are interested in whether, under stronger hypotheses than GRH, a better lower bound than \eqref{eq: AH FT} can be shown. This question was already taken up Juhas' thesis \cite{Ju16}, in which the following theorem was proved:  

\begin{Theorem}[Theorem 9 of \cite{Ju16}]
Assume that for all $X$, for all $1 \leqslant m \leqslant X^{1/2}$, and for all non-principal Dirichlet characters $\chi \, (\text{mod } m)$ one has 
\begin{equation}
\label{eq: Hypothesis M}
\vert \psi(X,\chi)\vert \ll_{\varepsilon} X^{1/2 + \varepsilon} m^{-1/4},
\end{equation} where $\psi(X,\chi) = \sum_{ n \leqslant X}\Lambda(n) \chi(n)$ as usual. (This is `Hypothesis M' in \cite{Ju16}.) Then, for all $\varepsilon >0$ one has\[ F_T(\alpha) \geqslant \frac{5}{4} - \frac{3}{4} \alpha - \varepsilon,\] provided $1 \leqslant \alpha \leqslant \frac{5}{3} - \varepsilon$ and $T \geqslant T_0(\varepsilon)$.
\end{Theorem}
\noindent Juhas' argument makes use of some ingenious extra averaging over $k$ for the correlations $\sum_n \lambda_z(n) \lambda_z(n+k)$ that is not present in \cite{GGOS00}  (see Remark \ref{Remark on the thesis} below). Unfortunately, the hypothesis (\ref{eq: Hypothesis M}) is actually false. This follows from Theorem 1.1 of the recent pre-print \cite{dlBF20} of de la Bret\`{e}che and Fiorilli\footnote{Our thanks to D. Fiorilli for making us aware of the work in \cite{dlBF20}} on lower bounds for the moments of primes in arithmetic progressions.  

Even without this recent development in \cite{dlBF20}, the bound \eqref{eq: Hypothesis M} would have contradicted several conjectures in the literature. For instance, \eqref{eq: Hypothesis M} would have implied that \[\sum_{Q < q \leqslant 2Q} \sum\limits_{ a \leqslant q} \Big\vert \psi(X; q, a) - \delta_{(a,q)= 1} \frac{X}{\varphi(q)}\Big\vert^2 \ll_{\varepsilon} Q^{1/2} X^{1 + \varepsilon}\] when $Q < X^{1/2}$, contradicting Conjecture 1.1 of Fiorilli's paper \cite{Fi15} regarding the range over which the asymptotic Barban--Davenport--Halberstam theorem can be expected to hold. If (\ref{eq: Hypothesis M}) had been true for $m \leqslant X^{2/3 + \delta}$ for some small positive $\delta$ then it would have contradicted what was known on GRH regarding the variance of the primes in arithmetic progressions (see expression (5) of \cite{Fi15} and the references therein). And finally, if (\ref{eq: Hypothesis M}) had been true for $m \leqslant X^{1/2 + \delta}$ for some small positive $\delta$ then it would have contradicted Theorem 3 of Friedlander-Goldston \cite{FG96}. All in all, we feel that despite the work of Juhas there remains much merit in improving the lower bound in Theorem \ref{Theorem GGOS} subject to a standard (and widely believed) conjecture.

Our bound will be conditional upon the following:
\begin{Conjecture}[Strong error term for primes in APs, $AP(\theta)$]
\label{Conjecture Montgomery APs}
Let $0<\theta \leqslant 1$. Then $AP(\theta)$ is the conjecture that for all $x\geqslant 1$, $q < x^{\theta}$, and $a \leqslant q$, \[ \psi(x;q,a) = 1_{(a,q) = 1} \frac{x}{\varphi(q)}+ O_{\varepsilon,\theta}\Big(\frac{x^{1/2 + \varepsilon}}{ q^{1/2}}\Big).\]
\end{Conjecture}

\noindent Although $AP(\theta)$ is of course far beyond what is known unconditionally, even on average over $q$, there is a history of this conjecture in the literature. With $\theta = 1$, $AP(\theta)$ was first formulated by Friedlander--Granville \cite[Conjecture 1(b)]{FG89}. It is often attributed to Montgomery, however, since in \cite[(15.9)]{Mo06} Montgomery had proposed a similar version (the same as $AP(1)$ except with an error term of $(x/q)^{1/2+\varepsilon}$). By using an argument based on the Maier matrix method, Friedlander--Granville showed that Montgomery's original conjecture was unreasonably strong for the very largest moduli $q$. They proposed $AP(1)$ as a sensible alternative, one which would nonetheless represent approximately square-root cancellation in the error term.

Our lower bound on $F_T(\alpha)$, conditional on $AP(\theta)$, is as follows:  

\begin{Theorem}[Lower bound on $F_T(\alpha)$]
\label{Theorem Main Theorem}
Assume Conjecture \ref{Conjecture Montgomery APs} for $\theta = 27/53$. Then for all for all $\varepsilon >0$, for all $\alpha \in [1,95/94 - \varepsilon)$, and for all $T \geqslant T_0(\varepsilon)$, we have \[F_T(\alpha) \geqslant 
\frac{127}{53} - \frac{100\alpha}{53} - \varepsilon. \]
\end{Theorem}
\noindent Beyond $\alpha = 95/94$, we are not able to offer an improvement over Theorem \ref{Theorem GGOS}. In fact, observe that when $\alpha = 95/94$ one has $\frac{127}{53} - \frac{100 \alpha}{53} = \frac{3}{2} - \alpha$, i.e. our bound collapses to the G-G--\"{O}--S bound at the end-point of the range. \\

On the one hand, one could consider it unsurprising that the assumption of a conjecture such as $AP(\theta)$, for some $\theta > 1/2$, could lead to a stronger conclusion than the assumption of GRH alone. However, as will become clear in Section \ref{Sec:method of GGOS} when we describe the method of G--G--\"{O}--S, several different terms naturally occur in this approach, and $AP(\theta)$ only helps us to estimate one of them.

One might nonetheless be concerned that $AP(\theta)$ lies deeper than Conjecture \ref{Conjecture Montgomery pair corr}, thus rendering `content-free' any partial result towards Conjecture \ref{Conjecture Montgomery pair corr} that is conditional upon $AP(\theta)$. Of course, given two conjectures which both lie far beyond the reach of the present field, it is hard to say categorically which one lies deeper than the other. However, we have a line of argument which suggests that assuming strong results on the distribution of primes in arithmetic progressions is not a ridiculous move. Indeed, all the prior results that establish Conjecture \ref{Conjecture Montgomery pair corr} for some range of $\alpha \geqslant 1$ have been contingent on one of two conjectures: either the authors assume an asymptotic for the variance of the prime counting function $\psi(X)$ in short intervals, which on RH turns out to be equivalent to Conjecture \ref{Conjecture Montgomery pair corr} (see Goldston--Montgomery \cite{GM87}); or the authors assume a power-saving in the correlations 
\begin{equation}
\label{eq: twin prime asymptotics}
\sum\limits_{ n \leqslant X} \Lambda(n) \Lambda(n+k) = \mathfrak{S}(k) X + O(X^\eta),
\end{equation} which, on RH, establishes Conjecture \ref{Conjecture Montgomery pair corr} in the range $\vert \alpha\vert < \eta^{-1}$. One may consult \cite[Example 4] {GG98} or \cite[Corollary 1.2]{C04}. Moreover, Theorem 3 of Montgomery--Soundararajan's paper \cite{MS04} is an example in which power-saving asymptotics for the full Hardy--Littlewood $k$-tuple conjecture are assumed in order to get some handle on the distribution of primes in short intervals (and hence on Conjecture \ref{Conjecture Montgomery pair corr}). So it seems reasonable to suggest that Conjecture \ref{Conjecture Montgomery pair corr} lies deeper than estimates such as (\ref{eq: twin prime asymptotics}). Now, according to the present state of the field, a parity-breaking conjecture such as the Hardy--Littlewood prime $k$-tuples conjecture is not implied by even the strongest of conjectures on the distribution of primes in arithmetic progressions. One may consider the classical Bombieri sieve \cite{Bo75}, say, for an example of this fact.

In summary, a `cheap' direct path from strong results on primes in arithmetic progressions to statements concerning the function $F_T(\alpha)$ does not seem to presently exist.  \\

The proof of Theorem \ref{Theorem Main Theorem} will be given in Section \ref{Sec:method of GGOS}. We do not offer any structural innovations on the technique of G--G--\"{O}--S from Theorem \ref{Theorem GGOS}; instead, our improvement comes from inputting our stronger auxiliary estimates into their scheme. \\

Our second application of Propositions \ref{Proposition lambdaQ correlation} and \ref{Proposition lambdaQ squared on RH} is to lower-bounding the variance of the number of primes in short intervals. This variance is closely related to the pair correlations of the zeros of the Riemann zeta function. As we have already mentioned, on RH the following conjecture is equivalent to Conjecture \ref{Conjecture Montgomery pair corr} (see Goldston--Montgomery \cite{GM87}): 

\begin{Conjecture}
\label{Conjecture: variance}
For all $\varepsilon >0$, if $h \leqslant X^{1- \varepsilon}$ then
\[\int\limits_{0}^X (\psi(x+h) - \psi(x) - h)^2 \, dx  = (1 + o_\varepsilon(1)) hX \log(X/h)\] as $X \rightarrow \infty$. 
\end{Conjecture}
\noindent Similar methods to those used to prove Theorem \ref{Theorem GGOS} can be used to provide lower bounds on this variance. As far as we are aware, the strongest result that is currently known, on GRH, is the following (due to Goldston--Yıldırım):
\begin{Theorem}[Theorem 1 of \cite{GY01} with $q = 1$]
\label{Theorem: GY}
 Assume the Generalised Riemann Hypothesis for Dirichlet $L$-functions. For all $\varepsilon >0$, if $1 \leqslant h \leqslant X^{1/3 - \varepsilon}$ then \[\int\limits_{0}^X (\psi(x+h) - \psi(x) - h)^2 \, dx \geqslant \Big( \frac{1}{2} - o_{\varepsilon}(1)\Big) hX \log(X/h^3)\] 
\end{Theorem}

\noindent On conjecture $AP(\theta)$, we are able to increase the lower bound in Theorem \ref{Theorem: GY} (at least for small $h$). 

\begin{Theorem}
\label{Theorem lower bound on variance}
Assume Conjecture \ref{Conjecture Montgomery APs} for $\theta = \frac{27}{53}$. Then, for all $\varepsilon >0$, if $1 \leqslant h \leqslant X^{\frac{1}{95}- \varepsilon}$ then \[ \int\limits_{X}^{2X} (\psi(x+h) - \psi(x) - h)^2 \, dx \geqslant \Big(\frac{27}{53} - o_{\varepsilon}(1)\Big) hX \log(X/h^{127/27}).\]
\end{Theorem}
\noindent Since $\frac{27}{53} > \frac{1}{2}$ this gives an improvement over Theorem \ref{Theorem: GY} for small $h$. However, the lower bound in Theorem \ref{Theorem lower bound on variance} collapses to the lower bound in Theorem \ref{Theorem: GY} at the end-point of the range of applicability, namely at $h= X^{1/95}$. 

 We stress again that there are several terms which occur in the method from \cite{GY01}, only one of which is directly improved by assuming $AP(\theta)$ over assuming GRH. \\

\noindent \textbf{Acknowledgements}: We would like to thank Thomas Bloom for making us aware of the work in \cite{GGOS00} and for many interesting conversations on the topic of $\lambda_z(n)$ and $F_T(\alpha)$. We would also like to thank Andrew Granville and Dimitris Koukoulopoulos for their comments and advice, and to J\"{o}ni Ter\"{a}v\"{a}inen for comments on an earlier draft of the manuscript. Thanks also to Daniel Fiorilli, \'{E}tienne Fouvry, Kaisa Matom\"{a}ki and Maksym Radziwi\l\l\ for helpful comments and suggestions following the online publication of the first pre-print version. An anonymous referee gave very useful corrections. While working on this paper, the author was supported by a postdoctoral research fellowship from the Centre de Recherches Mathématiques, a junior fellowship at Institut Mittag-Leffler, and a junior research fellowship at Trinity College Cambridge.  \\

\section{Proof of Proposition \ref{Proposition lambdaQ correlation}}
\label{Sec:offdiagonal}

We will prove the following general result concerning correlations of sieve weights.  

\begin{Proposition}
\label{Prop: general lambdaQ correlation}
Let $1 \leqslant \vert k\vert \leqslant X$ and let $X_1 = \max(0,-k)$ and $X_2 = \min(X,X-k)$. Let $2 \leqslant z \leqslant X$, and let $\rho^*_{z}:\mathbb{N} \longrightarrow \mathbb{R}$ be any function such that $\rho_{z}^*(d) = 0$ for $d > z$ and $ \vert \rho_{z}^*(d) \vert \leqslant B$ for all $d$. Let $\lambda^*_{z}(n) = \sum_{d \vert n} \rho_{z}^*(d)$. Then
\[\sum\limits_{X_1 < n \leqslant X_2} \lambda_z^*(n) \lambda_z^*(n+k) = (X- \vert k\vert) \sum\limits_{\substack{ d_1,d_2 \leqslant z \\ (d_1,d_2) \vert k}} \frac{ \rho^*_z(d_1)}{d_1} \frac{\rho^*_z(d_2)}{d_2} (d_1,d_2)+ O_\varepsilon(B^2\min(z^2, X^{\frac{47}{74} + \varepsilon}z^{\frac{53}{74}})).\]
\end{Proposition}

\begin{proof}[Proof of Proposition \ref{Proposition lambdaQ correlation} assuming Proposition \ref{Prop: general lambdaQ correlation}]
For all $d$ we have \[ \vert \rho_{z,d} \vert \leqslant L(z).\] Indeed, $\rho_{z,d}$ is only supported on square-free $d$, and for such $d$ we have
\begin{align}
\label{bound on rho}
L(z) = \sum\limits_{k \vert d} \sum\limits_{ \substack{ q \leqslant z \\ (q,d) = k}} \frac{\mu^2(q)}{\varphi(q)} = \sum\limits_{k \vert d} \frac{ \mu^2(k)}{\varphi(k)} \sum\limits_{ \substack{ q \leqslant z/k \\ (q,d) = 1}} \frac{ \mu^2(q)}{\varphi(q)} \geqslant \Big(\sum\limits_{k \vert d} \frac{ \mu^2(k)}{\varphi(k)} \Big) \Big( \sum\limits_{ \substack{ q \leqslant z/d \\ (q,d) = 1}} \frac{ \mu^2(q)}{\varphi(q)}\Big)  \nonumber\\
= \frac{d}{\varphi(d)}  \sum\limits_{ \substack{ q \leqslant z/d \\ (q,d) = 1}} \frac{ \mu^2(q)}{\varphi(q)} = \vert \rho_{z,d}\vert.
\end{align}
Now, from the definition of $\rho_{z,d}$ and swapping orders of summation, \[  \sum\limits_{\substack{ d_1,d_2 \leqslant z \\ (d_1,d_2) \vert k}} \frac{ \rho_{z,d_1}}{d_1} \frac{\rho_{z,d_2}}{d_2} (d_1,d_2) = \sum\limits_{q_1,q_2 \leqslant z} \frac{ \mu^2(q_1)}{\varphi(q_1)} \frac{ \mu^2(q_2)}{\varphi(q_2)} \sum\limits_{ \substack{ d_1 \vert q_1 \\ d_2 \vert q_2 \\ (d_1,d_2) \vert k}} \mu(d_1) \mu(d_2) (d_1,d_2).\] By expression (2.10) of Goldston's paper \cite{Go95}, this equals \begin{equation}
\label{singular series estimate}
\mathfrak{S}(k) + O\Big(\frac{k \tau(k)}{\varphi(k) z}\Big),
\end{equation} where $\mathfrak{S}(k)$ is as in (\ref{eq: singular series}). 

It remains to derive the final error term in Proposition \ref{Proposition lambdaQ correlation}. Taking $B \ll_{\varepsilon} X^{\varepsilon}$ and using Proposition \ref{Prop: general lambdaQ correlation} to generate an error of $X^{\frac{47}{74} + \varepsilon} z^{\frac{53}{74}}$, we generate one of the final error terms in Proposition \ref{Proposition lambdaQ correlation}. The alternative error of $O(z^2)$ in Proposition \ref{Proposition lambdaQ correlation} (note the absence of an $X^{\varepsilon}$ term) follows from a direct application of the `classical' bound recalled in Proposition \ref{Prop: classical off diagonal}. Taking the minimum of these two errors, Proposition \ref{Proposition lambdaQ correlation} follows.   
\end{proof}

\begin{proof}[Proof of Proposition \ref{Prop: general lambdaQ correlation}]

We begin with establishing the $O(B^2 z^2)$ error term. This follows from the classical approach. Indeed, expanding the sums and swapping the orders of summation, we get 
\begin{align*}
\sum\limits_{X_1 < n \leqslant X_2} \lambda_z^*(n) \lambda_z^*(n+k) &= \sum\limits_{\substack{d_1,d_2 \leqslant z}} \rho_z^*(d_1) \rho_z^*(d_2) \sum\limits_{\substack{X_1 < n \leqslant X_2 \\ d_1 \vert n \\ d_2 \vert n + k}} 1 \\
& = (X - \vert k\vert )\sum\limits_{\substack{d_1,d_2 \leqslant z \\ (d_1,d_2) \vert k}} \frac{\rho_z^*(d_1) \rho_z^*(d_2)}{[d_1,d_2]} + O\Big( \sum\limits_{\substack{d_1,d_2 \leqslant z \\ (d_1,d_2) \vert k}} \vert \rho_z^*(d_1) \vert \vert \rho_z^*(d_2)\vert\Big) \\
& = (X- \vert k\vert ) \sum\limits_{\substack{ d_1,d_2 \leqslant z \\ (d_1,d_2) \vert k}} \frac{ \rho_z^*(d_1) \rho_z^*(d_2)}{d_1d_2} (d_1,d_2) + O(B^2 z^2),
\end{align*}
as required. 

To generate the $O(B^2 X^{\frac{47}{74} + \varepsilon} z^{\frac{53}{74}})$ error term, we begin by introducing a smooth cut-off function. Let $\delta >0$. It is a well-known construction that there is a function $f_\delta \in C_c^\infty(\mathbb{R})$ such that $f_\delta(x) \in [0,1]$ for all $x$,  $ f_\delta(x) \equiv 1$ for all $ x \in [1/2,1]$, $f_\delta(x) \equiv 0$ for all $ x < 1/2 - \delta$ and for all $x > 1 + \delta$, and 
\begin{equation}
\label{equation smooth bound}
\vert \widehat{f_\delta}(\alpha)\vert \ll_K\delta^{-K} \vert \alpha \vert^{-K}
\end{equation} for all $\alpha \in \mathbb{R}$ and  natural numbers $K$. For instance, one can use the construction in Lemma 3.1 of \cite{Wa19}, adapted to the interval $[1/2,1]$ rather than $[0,1]$. 

Let $N$ be a dyadic scale for the variable $n$, such that $X_1 < N/2 < N \leqslant X_2$. Owing to the bound \[ \vert\lambda^*_{z}(n)\vert \leqslant B \tau(n),\] we have 
\begin{align}
\label{introducing smoothing}
\sum\limits_{ N/2 < n \leqslant N}\lambda_z^*(n) \lambda_z^*(n+k) = &\sum\limits_{n \in \mathbb{Z}}\lambda_z^*(n) \lambda_z^*(n+k) f_\delta(n/N) \nonumber \\&+ O\Big(B^2 \Big(\sum\limits_{\substack{N/2 - \delta N \leqslant n < N/2 \\ N< n \leqslant N + \delta N}} \tau(n) \tau(n+k) \Big)\Big).
\end{align}
The divisor sum correlation is at most \[ \sum\limits_{N/2 - \delta N \leqslant n < N/2} \tau(n)^2 + \sum\limits_{N/2 - \delta N + k < n \leqslant N/2 + k} \tau(n)^2,\] plus an analogous sum over the interval $N < n \leqslant N+ \delta N$. Since \[ \sum\limits_{ m \leqslant M} \tau(n)^2 = MP_3(\log M) + O_\varepsilon(M^{1/2 + \varepsilon}),\] for some polynomial $P_3$ of degree $3$, we observe that the error term in (\ref{introducing smoothing}) is at most \begin{equation}
\label{error from smoothing}
O_\varepsilon(B^2(\delta X (\log X)^3 + X^{1/2 + \varepsilon})).
\end{equation}

We proceed to the main term. We have 
\begin{align*}
\sum\limits_{n \in \mathbb{Z}} \lambda_z^*(n) \lambda_z^*(n+k) f_\delta(n/N) = \sum\limits_{ \substack{d_1 \leqslant z \\ d_2 \leqslant z \\(d_1,d_2) \vert k}} \rho_{z}^*(d_1) \rho_{z}^*(d_2) \sum\limits_{ \substack{ n \in \mathbb{Z} \\ d_1 \vert n \\ d_2 \vert n+k}} f_\delta(n/N) \\
= \sum\limits_{d \vert k } \sum\limits_{ \substack{d_1 \leqslant z/d \\ d_2 \leqslant z/d \\ (d_1,d_2) = 1}} \rho_{z}^*(dd_1) \rho_{z}^*(dd_2) \sum\limits_{ \substack{ n \in \mathbb{Z} \\ n \equiv (-k/d)\overline{d_1} \, \text{mod } d_2}} f_\delta\Big(\frac{d d_1 n}{N}\Big),
\end{align*}
\noindent where $\overline{d_1}$ refers to the multiplicative inverse of $d_1$ modulo $d_2$ that lies in the range $1 \leqslant \overline{d_1} \leqslant d_2 - 1$. The above is equal to 
\[ \sum\limits_{d \vert k }  \sum\limits_{ \substack{d_1 \leqslant z/d \\ d_2 \leqslant z/d \\ (d_1,d_2) = 1}} \rho_{z}^*(d d_1) \rho_{z}^*(dd_2) \sum\limits_{ l \in \mathbb{Z} } f_\delta \Big( l\frac{ d d_1d_2}{N} - \frac{k d_1 \overline{d_1}}{N}\Big),\] and applying the Poisson summation formula to the inner sum we obtain 
\begin{equation}
\label{eq:postpoisson}
N\sum\limits_{d \vert k } \frac{1}{d} \sum\limits_{ \substack{d_1 \leqslant z/d \\ d_2 \leqslant z/d \\ (d_1,d_2) = 1}} \frac{\rho_{z}^*(dd_1)}{d_1} \frac{\rho_{z}^*(dd_2)}{d_2} \sum\limits_{ c \in \mathbb{Z}} \widehat{f_\delta}\Big(c\frac{N}{dd_1 d_2}\Big) e\Big(-c\frac{k}{d}\frac{\overline{d_1}}{d_2}\Big).
\end{equation}  The term with $c = 0$ is equal to 
\begin{equation}
\label{eq: main c equals 0 term}
(1 + O(\delta)) \frac{N}{2} \sum\limits_{\substack{d_1 \leqslant z \\ d_2 \leqslant z \\ (d_1,d_2) \vert k}} \frac{ \rho_{z}^*(d_1)}{d_1} \frac{ \rho_{z}^*(d_2)}{d_2}(d_1,d_2),
\end{equation}
since $\widehat{f_\delta}(0) = \int f = 1/2 + O(\delta)$. 

Now we consider the terms with $c \neq 0$, which will go into the error. First we deal with the contribution from those $c$ with $\vert c\vert \geqslant dd_1d_2 N^{-1} \delta^{-1} X^{\varepsilon}$. For these terms, the bound \eqref{equation smooth bound} yields \[ \sum\limits_{\vert c \vert \geqslant dd_1d_2 N^{-1} \delta^{-1} X^{\varepsilon}} \vert \widehat{f_\delta}(c N/dd_1d_2)\vert \ll_{\varepsilon} X^{-10}\] and so the contribution to the final error term is negligible. To deal with the remaining terms, we split the variables $d_1$ and $d_2$ into dyadic ranges, seeking (for each fixed divisor $d \vert k$) an upper bound on \[ \frac{N}{d} \sum\limits_{ \substack{ c \in \mathbb{Z} \\ 1 \leqslant \vert c\vert \leqslant dM_1 M_2 N^{-1} \delta^{-1} X^{\varepsilon}}} \Big\vert \sum\limits_{ \substack{M_1/2 < d_1 \leqslant M_1 \\ M_2/2 < d_2 \leqslant M_2 \\ (d_1,d_2) = 1 }} \frac{ \rho_{z}^{*}(d_1)}{d_1} \frac{ \rho_{z}^{*}(d_2)}{d_2} \widehat{f_\delta}\Big(c \frac{N}{dd_1d_2}\Big)e\Big(- c \frac{k}{d} \frac{ \overline{d_1}}{d_2}\Big)\Big\vert. \]

Expanding the Fourier transform of $f_{\delta}$, this is bounded above by a constant times 
\begin{equation}
\label{eq: before Bettin Chandee}
M_1 M_2 \delta^{-1} X^{\varepsilon} \max\limits_{ \substack{ c \in \mathbb{Z} \\ 1 \leqslant \vert c \vert \leqslant d M_1 M_2 N^{-1} \delta^{-1} X^{\varepsilon}}} \max\limits_{ x \in [0,2]} \Big\vert \sum\limits_{ \substack{M_1/2 < d_1 \leqslant M_1 \\ M_2/2 < d_2 \leqslant M_2 \\ (d_1,d_2) = 1 }} \frac{ \rho_{z}^*(d_1)}{d_1} \frac{ \rho_{z}^*(d_2)}{d_2} e\Big(- c \frac{k}{d} \frac{ \overline{d_1}}{d_2} - cx \frac{N}{d d_1 d_2}\Big)\Big\vert.
\end{equation}
By considering complex conjugation, without loss of generality we may assume that $c$ is positive. 


Now we apply the bound of Bettin--Chandee on trilinear Kloosterman fractions. 

\begin{Theorem}[Bettin-Chandee, Remark 1, \cite{BC18}]
Let $\theta \in \mathbb{\mathbb{Z}} \setminus \{0\}$. Let $M_1,M_2,C$ be dyadic scales, and for each integer $c$ in the range $C/2 < c \leqslant C$ we let $g_{c,\theta} \in C^1(\mathbb{R}^2)$ be some continuously differentiable function. Then, for some arbitrary weight functions $ \alpha_{d_1}$, $\beta_{d_2}$, and $\nu_{c}$, let \[ \mathcal{B}_{g, \theta}(M_1,M_2, C) := \sum\limits_{\substack{ M_1/2 < d_1 \leqslant M_1 \\ M_2/2 < d_2 \leqslant M_2 \\ (d_1,d_2) = 1 \\  C/2 < c \leqslant C}} \alpha_{d_1} \beta_{d_2} \nu_{c} e\Big( \theta \frac{c \overline{d_1}}{d_2} + g_{c,\theta}(d_1,d_2)\Big).\] Suppose that for some parameter $K$, for all integers $c$ in the range $C/2<c \leqslant C$, for all $x \in (M_1/2,M_1]$ and for all $y \in (M_2/2,M_2]$ we have \[ \frac{\partial}{\partial x} g_{c,\theta}(x,y) \ll \frac{K}{x^2 y}, \qquad \text{and } \qquad \frac{\partial}{\partial y} g_{c,\theta}(x,y) \ll \frac{K}{x y^2}.\] Finally, let \[ \Vert \alpha \Vert = \Big(\sum\limits_{ M_1/2 < d_1 \leqslant M_1} \vert \alpha_{d_1}\vert^2  \Big)^{1/2},\] and similarly for the other weights. Let $\varepsilon >0$. Then $\mathcal{B}_{g,\theta}(M_1,M_2,C)$ is at most a constant (depending on $\varepsilon$) times
\[\Vert \alpha\Vert \Vert \beta\Vert \Vert \nu\Vert \Big(1 + \frac{ \vert \theta \vert C + K}{M_1M_2}\Big)^{1/2} ((CM_1M_2)^{\frac{7}{20} + \varepsilon}(M_1 + M_2)^{\frac{1}{4}} + (CM_1 M_2)^{\frac{3}{8} + \varepsilon}(CM_1 + CM_2)^{\frac{1}{8}}).\] 
\end{Theorem}

To apply this bound in our setting, we let \[ \alpha_{d_1}: = \frac{\rho_{z}^*(d_1)}{d_1}, \quad \beta_{d_2} : = \frac{ \rho_{z}^*(d_2)}{d_2}, \qquad \theta = -\frac{k}{d}, \qquad g_{c,\theta}(d_1,d_2) = -cx \frac{N}{d d_1 d_2},\] and we take $C = c $ and the weight $\nu_c$ to be supported at a single element $c$ (and equal to $1$ there). Observe that we may take $K= cN/d$. Since $c \leqslant d M_1 M_2 N^{-1} \delta^{-1} X^{\varepsilon}$, $\Vert \alpha\Vert \ll B M_1^{-\frac{1}{2}}$, and $\Vert \beta \Vert \ll B M_2^{-\frac{1}{2}}$, we infer that (\ref{eq: before Bettin Chandee}) is at most a constant (depending on $\varepsilon$) times  
\begin{align}
\label{eq: contributions from c neq 0}
B^2 M_1^{1/2} M_2^{1/2} \delta^{-1}X^{\varepsilon}\Big( 1 + \delta^{-1}  \Big( 1 + \frac{\vert k\vert}{N}\Big)\Big)^{\frac{1}{2}} \Big(d^{\frac{7}{20}}&\delta^{-\frac{7}{20} }N^{-\frac{7}{20} }(M_1 ^{\frac{19}{20}} M_2 ^{\frac{7}{10} } + M_1^{\frac{7}{10}}   M_2 ^{\frac{19}{20}})\nonumber \\& + d^{\frac{1}{2}} \delta^{-\frac{1}{2} } N^{- \frac{1}{2} }(M_1  M_2^{\frac{7}{8}} + M_1^{ \frac{7}{8}} M_2)\Big),
\end{align}
provided $\delta \geqslant X^{-O(1)}$. What remains is the task of adding together the contributions from various dyadic scales, and optimising for $\delta$. 

Without loss of generality we assume $k >0$ (the case $k <0$ is symmetric), so $X_1 = 0$ and $X_2 = X-k$. Let $\eta \in (0,1)$ be a positive parameter to be optimised later. We let $N_l = 2^{-l + 1}(X-k)$, for $1 \leqslant l \leqslant \log (X^{\eta})/(\log 2)$. Then \[ \sum\limits_{ 0 < n \leqslant X-k} \lambda_z^*(n) \lambda_z^*(n+k) = \sum\limits_{ 1 \leqslant l  \leqslant (\log (X^{\eta}))/(\log 2)} \sum\limits_{ N_l/2 < n \leqslant N_l} \lambda_z^*(n) \lambda_z^*(n+k) + O_{\varepsilon}(B^2X^{1-\eta + \varepsilon}).\] Adding together the main terms from (\ref{eq: main c equals 0 term}), we get \[ (1+O(\delta))(X-k)\Big(\sum\limits_{1 \leqslant l \leqslant \log (X^{\eta})/\log 2} 2^{-l}\Big) \Big(\sum\limits_{\substack{d_1,d_2 \leqslant z \\(d_1,d_2) \vert k}} \frac{\rho_z^*(d_1)}{d_1} \frac{ \rho_z^*(d_2)}{d_2} (d_1,d_2)\Big),\] which is \[ (X-k) \Big(\sum\limits_{\substack{d_1,d_2 \leqslant z \\(d_1,d_2) \vert k}} \frac{\rho_z^*(d_1)}{d_1} \frac{ \rho_z^*(d_2)}{d_2} (d_1,d_2)\Big)+ O_\varepsilon(B^2 (\log z)^2 X^{\varepsilon}( X^{1 - \eta} + \delta X)).\] This has the desired main term. 

Now we need to consider those terms with $c \neq 0$, i.e. expression (\ref{eq: contributions from c neq 0}), summed over $d \vert k$ (and dyadic ranges $N$). By considering the worst-case scenario ($k = X$, $N = X^{1 - \eta}$, $M_1 = M_2 = z/d$), one may upper-bound the contribution to \eqref{eq: contributions from c neq 0} by \[O_{\varepsilon}(X^{\varepsilon}B^2 \delta^{-\frac{3}{2}}X^{\frac{\eta}{2}}( z^{\frac{53}{20}} d^{-\frac{46}{20}} \delta^{-\frac{7}{20}} X^{-\frac{7(1 - \eta)}{20}} + z^{\frac{23}{8}} d^{-\frac{19}{8}} \delta^{-\frac{1}{2}} X^{-\frac{(1 - \eta)}{2}})).\] Summing over $d \vert k$, and combining all errors, we arrive at a total error of \[ O_{\varepsilon}(X^{\varepsilon} B^2(\delta X + X^{1 - \eta} + X^{1/2} + z^{\frac{53}{20}} X^{\frac{17\eta}{20} - \frac{7}{20}} \delta^{-\frac{37}{20}} + z^{\frac{23}{8}} X^{\eta - \frac{1}{2}} \delta^{-2})).\] 

After a tedious but straightforward calculation, when $\frac{\log z}{\log X} < \frac{27}{53}$ this error term is minimised when $\delta = X^{-\eta}$ and 
\begin{equation}
\label{eta cases}
\eta = \frac{27}{74} - \frac{53}{74} \frac{\log z}{\log X},
\end{equation}
with the dominant error term coming from the $\delta X$, $X^{1- \eta}$, and $z^{\frac{53}{20}} X^{\frac{17 \eta}{20} - \frac{7}{20}} \delta^{-\frac{37}{20}}$ contributions (which are all equal). Substituting the above value of $\eta$ into the error term $X^{1 - \eta}$ yields an overall error term of $O_{\varepsilon}(B^2 X^{\frac{47}{74} + \varepsilon} z^{\frac{53}{74}})$.  

When $1 \geqslant \frac{\log z}{\log X} \geqslant \frac{27}{53}$, we have $B^2 X^{\frac{47}{74} + \varepsilon} z^{\frac{53}{74}}\gg_{\varepsilon} B^2 X^{1 + \varepsilon}$, which is 
trviailly a valid error term in Proposition \ref{Prop: general lambdaQ correlation} (for trivial reasons). Therefore, an overall error of\\ $O_{\varepsilon}(B^2 \min(z^2, X^{\frac{47}{74} + \varepsilon} z^{\frac{53}{74}}))$ holds for all $z \leqslant X$, and Proposition \ref{Prop: general lambdaQ correlation} is proved. 
\end{proof}

\begin{Remark}
\label{Remark on the thesis}
\emph{For the intended application to the pair correlation of zeros of the Riemann zeta function, one has some extra averaging over $k$ for the correlations $\lambda_z(n) \lambda_z(n+k)$. Some authors have taken advantage of this, and given stronger estimates in certain ranges. For example, Lemma 1.11 of \cite{Ju16} is a bound for the sum \[ \sum\limits_{ k \leqslant h} (h-k) \sum\limits_{n \leqslant X} \lambda_z(n) \lambda_z(n+k)\] with a dominant error term of $O_{\varepsilon}(h^{3/2} z^2 X^{\varepsilon})$. For large $h$ this is a stronger result than what may be proved by applying Proposition \ref{Proposition lambdaQ correlation} to the inner sum and summing the error terms na\"{i}vely.}
\end{Remark}
\begin{Remark}
\emph{Some not-dissimilar averages to \eqref{eq:postpoisson} occur in recent work\footnote{Our thanks to K. Matom\"{a}ki for making us aware of this relationship.} of Matom\"{a}ki on finding almost primes in almost all very short intervals (see the statement of \cite[Lemma 4.3]{Ma20}). Here the celebrated bounds of Deshouillers--Iwaniec \cite{DI82} on averages of Kloosterman sums were used to obtain a saving. Unfortunately this approach relies not only on having extra averaging over $k$, as in Remark \ref{Remark on the thesis}, but on having a genuinely trilinear sum in place of the bilinear sum \eqref{eq:postpoisson}. This technique may be powerfully applied to the case when the sieve weight $\rho_z^*$ is `well-factorable', as in \cite{Ma20}, but in the applications of this paper we are dealing with Selberg weights, which are not well-factorable. In Section \ref{Sec: AH} we will discuss the options for future work involving a change of weight.}
\end{Remark}

\section{Proof of Proposition \ref{Proposition lambdaQ squared on RH}}
\label{Sec:Hooley method}
As we mentioned in the introduction, much of what we do in this section was first accomplished by Hooley in \cite[Lemma 2]{Ho00}. That said, it does not seem to be too much of a stretch to remark that \cite[Lemma 2]{Ho00} is a little buried in the literature, so it seems worth reviving. We hope too that we may provide the reader with some motivation as to why this approach might be expected to succeed, which is not included in \cite{Ho00}.

 We let\[ \mathbb{D}_z(s) := \sum\limits_{n=1}^\infty \frac{\lambda_z^2(n)}{n^s}\] be the Dirichlet series associated to $\lambda_z^2$. Since $\vert \rho_{z,d} \vert \leqslant L(z)$ (as proved in \eqref{bound on rho} above), we have
\begin{equation}
\vert \lambda_z(n) \vert \leqslant L(z) \tau(n). 
\end{equation}
Therefore the series for $\mathbb{D}_z(s)$ is absolutely convergent in the region $\Re s >1$. Swapping sums reveals that \[ \mathbb{D}_z(s) = \zeta(s) \sum\limits_{ d_1,d_2 \leqslant z} \frac{\rho_{z,d_1} \rho_{z,d_2}}{[d_1,d_2]^s},\] and so $\mathbb{D}_z(s)$ has a meromorphic continuation to the whole of the complex plane. We write \[ \mathbb{F}_z(s) : = \sum\limits_{ d_1,d_2 \leqslant z} \frac{\rho_{z,d_1} \rho_{z,d_2}}{[d_1,d_2]^s},\] and a key element of the proof will be establishing a non-trivial bound on $\vert \mathbb{F}_z(s)\vert$ to the left of the $1$-line.
\begin{Lemma}
\label{Lemma bound on FQ}
Assume RH and let $\varepsilon \in (0,1/10)$. Denoting $s =\sigma + it$, if $\sigma > 1/2 + \varepsilon$ then for all $\delta >0$ \[ \vert \mathbb{F}_z(s)\vert  = O_{\delta,\varepsilon}( (\vert t\vert + 2)^{\delta}z^{1 - \sigma + \delta}).\]
\end{Lemma}
\begin{proof}
Following Hooley, in particular \cite[expression (43)]{Ho00}, from combinatorial manipulation one derives 
\begin{equation}
\label{FQmanipulation}
 \mathbb{F}_z(s) = \sum\limits_{ u \leqslant z} \frac{ \varphi(u,s) \mu^2(u)}{u^{2s}} \Big( \sum\limits_{ \substack{d \leqslant z/u \\ (d,u) = 1}} \frac{ \rho_{z,du}}{d^s}\Big)^2,
 \end{equation} where \[\varphi(u,s): = \prod\limits_{ p \vert u}(p^s - 1).\] Let us now describe a heuristic for the remainder of the argument. If we were to undertake a na\"{i}ve approach, bounding the inner sum of (\ref{FQmanipulation}) trivially, then one would derive
\begin{align*}
\vert\mathbb{F}_z(s)\vert &\ll (\log z)^{2} \sum\limits_{ u \leqslant z} \frac{\mu^2(u)}{u^{\sigma}}\Big(\sum\limits_{ d \leqslant z/u} \frac{1}{d^{\sigma}}\Big)^2 \\
& \ll (\log z)^{2} \sum\limits_{ u \leqslant z} \frac{\mu^2(u)}{u^{\sigma}} \Big(\frac{z}{u}\Big)^{2(1-\sigma)} \\
& \ll (\log z)^{2} z^{2(1- \sigma)}.
\end{align*}
\noindent The exponent $2$ for the power of $z^{1 - \sigma}$ would lead, ultimately, to an error term in Proposition \ref{Proposition lambdaQ squared on RH} involving $z^2$, with no hope of improving upon \eqref{equation classical estimate}. 

However, since $L(z)^{-1}\rho_{z,du}$ is a sieve weight we might reasonably expect to be able to approximate the behaviour of $L(z)^{-1}\rho_{z,du}$ by the behaviour of $\mu(du)$. Were we to make this replacement, the inner Dirichlet polynomial in (\ref{FQmanipulation}) would become \[ L(z)\mu(u)\sum\limits_{\substack{d \leqslant z/u \\ (d,u) = 1}} \frac{\mu(d)}{d^s}.\] Using RH, the Dirichlet polynomial $\sum_{d \leqslant D} \mu(d) d^{-s}$ may in turn be approximated by $\zeta(s)^{-1}$. Since RH also implies that $\vert \zeta(s)\vert^{-1} \ll_{\delta, \varepsilon}
 (\vert t\vert + 2)^{\delta}$ in this range, this leads to the agreable \begin{align*}
\vert \mathbb{F}_z(s)\vert &\ll_{\delta, \varepsilon} (\vert t\vert + 2)^{\delta}(\log z)^2 \sum\limits_{u \leqslant z}\frac{\mu^2(u)}{u^\sigma} \\
&\ll_{\delta, \varepsilon}(\vert t\vert + 2)^{\delta}z^{1- \sigma}(\log z)^2 .
\end{align*}
This would settle Lemma \ref{Lemma bound on FQ}.

What remains is for us to make these approximations precise. The main auxiliary lemma is the following.

\begin{Lemma}
\label{Lemma bound on M}
Assume RH. Let $D \geqslant 1$ and $\varepsilon \in (0,1/10)$ be reals, and let $k$ be a natural number. Define \[ M_{D,k}(s) := \sum\limits_{ \substack{ d \leqslant D \\ (d,k) = 1}} \frac{ \mu(d) d}{\varphi(d) d^s}.\] Then, if $s = \sigma + it$ with $\frac{1}{2}+ \varepsilon< \sigma < 1$, for all $\delta  >0$ we have \[ \vert M_{D,k}(s) \vert  = O_{\delta, \varepsilon}((\vert t\vert + 2)^{\delta}k^{\delta}).\]
\end{Lemma}
\noindent The error term is uniform in the choice of $D$. 
\begin{proof}
Without loss of generality we may assume that $\delta$ is small enough in terms of $\varepsilon$. Let \[ G_{k}(w) : = \sum\limits_{ \substack{ d \geqslant 1 \\ (d,k) = 1}} \frac{\mu(d) d}{\varphi(d) d^w}.\] The series defining $G_k(w)$ is absolutely convergent in the region $\Re w > 1$. Following Hooley \cite[p. 63]{Ho00}, from elementary manipulation of the Euler product we have (when $\Re w > 1$) \[ G_k(w) = \begin{cases}
\frac{1}{\zeta(w)} H_k(w) B(w) (1 + \frac{1}{2^w - 1}) & \text{ if } 2 \vert k\\
\frac{1}{\zeta(w)} H_k(w) B(w) (1 - \frac{1}{2^w - 1}) & \text{ if } 2 \nmid k, \end{cases}\] where\footnote{We thank the anonymous referee for pointing out a typo in Hooley at this point, where the $-1$ exponent is neglected in the definition of $H_k$. Consequently we've also had to correct the Euler factor corresponding to $p=2$.} \[ H_k(w): = \prod\limits_{ \substack{p \vert k \\ p \geqslant 3}} \Big( 1 - \Big(1 - \frac{1}{p}\Big)^{-1} \frac{1}{p^w}\Big)^{-1}\] and \[ B(w): = \prod\limits_{p \geqslant 3} \Big( 1 - \frac{1}{p^w (p-1)} \Big( 1 - \frac{1}{p^w}\Big)^{-1}\Big).\] Since we are assuming RH, $\zeta(w)^{-1}$ has an analytic continuation to the region $\Re w > \frac{1}{2} + \frac{\varepsilon}{2}$. Furthermore, $H_k(w)$, $B(w)$, $(1 + \frac{1}{2^{w} - 1})$ and $(1 - \frac{1}{2^w - 1})$ define holomorphic functions in the same region (as they define absolutely uniformly convergent Euler products there). By taking logarithms and Taylor expanding, in this same region we have the bound
\begin{align*}
\vert H_k(w)\vert \ll \exp(\sum\limits_{p \vert k} \frac{1}{p^{\Re w}}) \ll \exp(\sum\limits_{p \leqslant \omega(k)} \frac{1}{p^{\frac{1}{2}}}) \ll \exp(O(\omega(k)^{\frac{1}{2}})) \ll \exp(O((\log k)^{\frac{1}{2}})) \ll_{\delta} k^{\delta},
\end{align*}
\noindent where $\omega(k)$ is the number of distinct prime factors of $k$. Also, \[ \vert B(w)\vert = O_{\varepsilon}(1).\] Finally, by RH, in this region we have \[ \Big\vert \frac{1}{\zeta(w)}\Big\vert = O_{\varepsilon,\delta}((\vert \Im w\vert + 2)^\delta),\] see \cite[Theorem 14.2]{Ti86}. Therefore, we have the overall bound
\begin{equation}
\label{eq: bound on Gz}
\vert G_k(w) \vert  = O_{\varepsilon,\delta}((\vert \Im w \vert + 2)^\delta k^{\delta}).
\end{equation}

Now, pick $c$ such that $1 + \delta < \sigma + c < 2$ and, without loss of generality, assume that $D$ is a half-integer. (Indeed, shifting the range $D$ by up to $\frac{1}{2}$ would introduce an overall error of $O(1)$ to the value of $M_{D,k}(s)$, which may be absorbed into the claimed bound on $\vert M_{D,k}(s)\vert$ in the statement of the lemma). Then, by Perron's formula (see \cite[Lemma 3.12]{Ti86} for this particular formulation), for a large $T$ to be chosen later \[ M_{D,k}(s) = \frac{1}{2 \pi i} \int\limits_{ c - iT}^{c + iT} G_k(s + s^\prime) \frac{D^{s^\prime}}{s^\prime} \, ds^\prime + O_{\delta}(T^{-1}(D^{c} + D^{1 - \sigma + \delta} )).\] Now we move the contour $\Re s^\prime = c$ to the line $\Re s^\prime = -\frac{\varepsilon}{2}$. There is a main-term contribution of $G_k(s)$, coming from the pole at $s^{\prime} = 0$. We bound this directly using \eqref{eq: bound on Gz}, giving a contribution of $\ll_{\varepsilon, \delta}(\vert t\vert + 2)^{\delta}k^{\delta}$ as required. 

Regarding the error term, we get the usual sum of three integrals (one along a vertical contour with $\Re s^\prime = -\frac{\varepsilon}{2}$, the other two along horiztonal contours with $\vert \Im s^\prime \vert = T$). Note that $\frac{1}{2} + \frac{\varepsilon}{2} < \Re s - \frac{\varepsilon}{2}$, and so using the bound (\ref{eq: bound on Gz}) for $\vert G_{k}(w)\vert$ in this region we get a combined error term of 
\begin{align*}
O_{\varepsilon,\delta}(k^{\delta}D^{-\varepsilon/2} (T^\delta + (\vert t \vert +2)^{\delta})) & + O_{\varepsilon,\delta}(k^{\delta} D^{c} T^{-1}(T^\delta + (\vert t \vert +2)^{\delta}))\\& + O_\delta(T^{-1}(D^{c} + D^{1 - \sigma + \delta})).
\end{align*} The above expression is inferior to the claimed error in the lemma if we choose $T = D^{10}$ and $\delta$ is small enough depending on $\varepsilon$.  
\end{proof}

To use Lemma \ref{Lemma bound on M}, we write \[ \Gamma_{z}(u,s): = \sum\limits_{ \substack{d \leqslant z/u \\ (d,u) = 1}} \frac{\rho_{z,du}}{d^s},\] so that from (\ref{FQmanipulation}) we get 
\begin{equation}
\label{FQmanipulation2}
 \mathbb{F}_z(s) = \sum\limits_{ u \leqslant z} \frac{ \varphi(u,s) \mu^2(u)}{u^{2s}}\Gamma_{z}(u,s)^2.
\end{equation}
Then, following Hooley \cite[expression (45)]{Ho00}, by expanding the definition of $\rho_{z,du}$ we establish 
\begin{align}
\label{Gamma expanded}
\Gamma_{z}(u,s) &= \frac{u\mu(u)}{\varphi(u)} \sum\limits_{ \substack{u^\prime \leqslant z/u \\ (u^\prime,u) = 1}} \frac{\mu^2(u^\prime)}{\varphi(u^\prime)} \sum\limits_{\substack{ d \leqslant z/uu^\prime \\ (d,uu^\prime) = 1}} \frac{\mu(d) d}{\varphi(d) d^s} \nonumber \\
& = \frac{u\mu(u)}{\varphi(u)} \sum\limits_{ \substack{u^\prime \leqslant z/u \\ (u^\prime,u) = 1}} \frac{\mu^2(u^\prime)}{\varphi(u^\prime)}  M_{z/u u^\prime, uu^\prime}(s).
\end{align}
By Lemma \ref{Lemma bound on M}, we conclude that if $\frac{1}{2}+ \varepsilon < \sigma < 1$ then \[\vert \Gamma_{z}(u,s)\vert = O_{\varepsilon,\delta}((\vert t\vert +2)^{\delta} z^{\delta} u\varphi(u)^{-1}).\] Plugging this bound into (\ref{FQmanipulation2}), if $\frac{1}{2} + \varepsilon < \sigma < 1$ then \[\vert \mathbb{F}_z(s)\vert \ll_{\varepsilon, \delta} z^{\delta} (\vert t \vert + 2)^{\delta}\sum\limits_{u \leqslant z}\frac{u^2}{u^{\sigma} \varphi(u)^2} \ll_{\varepsilon, \delta}(\vert t \vert + 2)^{\delta} z^{1 - \sigma +\delta}.\] Lemma \ref{Lemma bound on FQ} is proved. 
\end{proof}

Proposition \ref{Proposition lambdaQ squared on RH} follows quickly. Indeed, pick some $c \in (1,2]$. Then, by Perron's formula \cite[Lemma 3.12]{Ti86} again, we have 
\begin{align*}
&\sum\limits_{ n \leqslant X}\lambda_z^2(n) \\
&= \frac{1}{2 \pi i} \int\limits_{c - i T}^{c+ i T} \mathbb{D}_z(s) \frac{X^{s}}{s} \, ds + O_{\delta}\Big( (\log z)^{2}X^{\delta}1_{X \in \mathbb{N}} + X^c \sum\limits_{ n \geqslant 1} \frac{\lambda_z^2(n)}{n^c} \min \Big( 1, \frac{1}{T \vert \log (X/n) \vert }\Big)\Big) \\
& = \frac{1}{2 \pi i} \int\limits_{c - i T} ^{c + i T} \zeta(s) \mathbb{F}_z(s)\frac{X^{s}}{s} \, ds + O_{\delta}( X^{c+ \delta}T^{-1} + X^{\delta}).
\end{align*}

Moving the contour to the line $\Re s = \frac{1}{2} + \varepsilon$, for some sufficiently small $\varepsilon$, we pick up a main term of $X \mathbb{F}_z(1)$ from the pole at $s=1$. It is a standard calculation in the theory of the Selberg sieve that 
\begin{equation}
\label{equation Selberg sieve main term}
\mathbb{F}_z(1) = \sum\limits_{d_1,d_2 \leqslant z} \frac{\rho_{z,d_1} \rho_{z,d_2}}{[d_1,d_2]} = L(z),
\end{equation}
\noindent see \cite[Exercise 9.3.6]{Mu06}. Therefore the main term is $X L(z)$.

Regarding the error terms, we get the usual three integrals: one up a vertical contour from $\frac{1}{2} + \varepsilon - iT$ to $\frac{1}{2} + \varepsilon + iT$, and two along horizonal contours (from $\frac{1}{2} + \varepsilon - iT$ to $c - iT$, and from $\frac{1}{2} + \varepsilon + iT$ to $c + iT$). The bound from Lemma \ref{Lemma bound on FQ} together with the Lindel\"{o}f bound
\[\vert \zeta(s)\vert \ll_{\delta} (\vert t\vert + 2)^{\delta}\] when $\Re s = \frac{1}{2} + \varepsilon$, which follows from RH, gives \[ \Big\vert \int\limits_{ 1/2 + \varepsilon + iT}^{1/2 + \varepsilon - iT} \frac{\zeta(s) \mathbb{F}_z(s) X^s}{s} ds\Big\vert = O_{\varepsilon, \delta}(X^{\frac{1}{2} + \varepsilon} z^{\frac{1}{2} - \varepsilon + \delta} T^{\delta}). \] For the horizontal integeral, we can use a trivial bound \[ \vert\mathbb{F}_z(s)\vert \ll \sum\limits_{d_1,d_2 \leqslant z} \Big\vert\frac{ \rho_{z,d_1} \rho_{z,d_2}}{[d_1,d_2]^s}\Big\vert \ll (\log z)^2 \sum\limits_{d_1,d_2 \leqslant z} \frac{1}{[d_1,d_2]^{\Re s}} \ll z^2(\log z)^2\] when $\Re s > 0$. Together with Lindel\"{o}f, we obtain\[ \Big\vert \int\limits_{\frac{1}{2} + \varepsilon + iT}^{c + iT} \frac{ \zeta(s) \mathbb{F}_z(s) X^s}{s} \, ds \Big\vert \ll_{\delta} c X^c z^2(\log z)^2 T^{-1 + \delta},\] and the same for the other horiztonal integral. Since $c = O(1)$ is fixed and $z \leqslant X$, we have a totall error of \[O_{\delta}(X^{c + 2 + \delta} T^{-1+\delta}) + O_{\varepsilon, \delta}(X^{\frac{1}{2} + \varepsilon} z^{\frac{1}{2} - \varepsilon + \delta} T^{\delta}).\] 

Picking $T = X^{100 c}$ say, and $\delta = \frac{\varepsilon}{100 c}$, we obtain \begin{equation}
\label{eq:theboundabove}
\sum\limits_{ n \leqslant X}\lambda_z^2(n) = XL(z) + O_{\varepsilon}(X^{\frac{1}{2} + 3\varepsilon} z^{\frac{1}{2}}).
\end{equation}
Since $\varepsilon$ is arbitrary, Proposition \ref{Proposition lambdaQ squared on RH}  follows after combining \eqref{eq:theboundabove} with the classical bound (\ref{equation classical estimate}). \qed \\

\section{Proof of applications}
\label{Sec:method of GGOS}
In this section we discuss the proof of Theorems \ref{Theorem Main Theorem} and \ref{Theorem lower bound on variance}. For experts, we can say that our proof consists of nothing more than substituting our stronger auxiliary estimates on $\lambda_z(n)$ into the general framework provided by G--G--\"{O}--S \cite{GGOS00} and Goldston--Yildrim \cite{GY01}. We have no desire to repeat large tracts of the existing literature, but similarly it seems remiss not to give the general reader at least some indication of how our Propositions \ref{Proposition lambdaQ correlation} and \ref{Proposition lambdaQ squared on RH} are used. What follows is, we hope, an informative balance. \\

We focus on Theorem \ref{Theorem Main Theorem}. Fix $ \alpha \in [1,2)$, let $X = T^\alpha$, and let \[ F(X,T): = \Big(\frac{T}{2 \pi} \log T \Big) F_T(\alpha).\] Since we are assuming $AP(27/53)$ we may also assume RH, as RH is equivalent to the $q=1$ form of Conjecture \ref{Conjecture Montgomery APs}. Let $\Psi_U: \mathbb{R} \longrightarrow [0,1]$ be a smooth approximation to the indicator function of the interval $[0,1]$, supported on $[-1/U, 1 + 1/U]$, with the properties that $\Psi_U(t) \in [0,1]$ and $\Psi^{(j)}(t) \ll U^j$ for $j=1,2,3...$. (We use this notation here in order to match the notation in \cite{GGOS00}, to allow for easier comparisons.) We begin with the following lemma of G--G--\"{O}--S, which is proved using the same technique of Montgomery \cite{Mo73}:

\begin{Lemma}[Lemma 1, \cite{GGOS00}]
\label{Lemma GGOS lemma}
Assume RH. If $X \gg T$ and $U = \log ^B T$ with $B>1$, then we have 
\begin{equation}
F(X,T) = \frac{1}{2 \pi X^2} I_1(X,T) + \frac{X^2}{2 \pi} I_2(X,T) + O\Big(\frac{T \log^2 T}{U}\Big) + O_{\varepsilon}\Big( \frac{X^{1 + \varepsilon}}{T}\Big),
\end{equation}
\noindent where 
\begin{equation}
I_1(X,T) = \int\limits_{-\infty}^\infty \Psi_U\Big(\frac{t}{T}\Big) \Big\vert \sum\limits_{ n \leqslant X} \Lambda(n) n^{\frac{1}{2} - it} - \int\limits_{1}^X u^{1/2 - it} \, du \Big\vert^2 \, dt
\end{equation}
\noindent and 
\begin{equation}
I_2(X,T) = \int\limits_{-\infty}^\infty \Psi_U\Big(\frac{t}{T}\Big) \Big\vert \sum\limits_{n > X} \frac{\Lambda(n)}{n^{3/2 + it}} - \int\limits_{X}^\infty u^{-3/2 - it} \, du \Big\vert^2 \, dt.
\end{equation}
\end{Lemma}

The action then shifts to estimating the integrals $I_1(X,T)$ and $I_2(X,T)$. If one expands the definitions, one quickly confronts the problem of having to control the values of the correlations $\sum_{n \leqslant N}\Lambda(n)\Lambda(n+k)$. To get around this issue, the idea of G--G--\"{O}--S is to note that $I_1(X,T)$ is at least 
\begin{align*}
\int\limits_{-\infty}^\infty \Psi_U\Big(\frac{t}{T}\Big) \Big\vert \sum\limits_{ n \leqslant X} \Lambda(n) n^{\frac{1}{2} - it} - &\int\limits_{1}^X u^{1/2 - it} \, du \Big\vert^2 \, dt \, - \\
&\int\limits_{-\infty}^\infty \Psi_U\Big(\frac{t}{T}\Big) \Big\vert \sum\limits_{n \leqslant X} \lambda_z(n) n^{\frac{1}{2} - it}  -\sum\limits_{ n \leqslant X} \Lambda(n) n^{\frac{1}{2} - it} du \Big\vert^2 \, dt,
\end{align*} and therefore, by expanding out the squares and rearranging the terms, to conclude that $I_1(X,T)$ is at least
\begin{align}
\label{eq: lower bound on I1}
 2 \Re\int\limits_{ -\infty}^{\infty} \Psi_U\Big(\frac{t}{T}\Big) \Big( \sum\limits_{n \leqslant X} \Lambda(n) n^{1/2 - it} - \int\limits_{1}^X u^{1/2 - it} \, du \Big) \Big( \overline{ \sum\limits_{n \leqslant X} \lambda_z(n) n^{1/2 - it} - \int \limits_{1}^X u^{1/2 - it} \, du }\Big) \, dt \nonumber\\
- \int\limits_{-\infty}^{\infty} \Psi_U\Big(\frac{t}{T}\Big) \Big\vert \sum\limits_{n \leqslant X}\lambda_z(n) n^{1/2 - it} - \int\limits_{1}^X u^{1/2 - it} \, du \Big\vert^2 \, dt.
\end{align}
\noindent There is a similar expression for $I_2(X,T)$. Expanding out the brackets, one sees that the relevant correlations are now $\sum_n\Lambda(n)\lambda_z(n)$, $\sum_n\Lambda(n)\lambda_z(n+k)$, $\sum_n\lambda_z(n)^2$, and $\sum_n\lambda_z(n)\lambda_z(n+k)$, which promise to be more easily tractable than $\sum_n\Lambda(n) \Lambda(n+k)$. 

For example, one term in the expansion is \[ \int\limits_{-\infty}^{\infty} \Psi_U\Big( \frac{t}{T}\Big) \sum\limits_{m,n \leqslant X} \Lambda(n) \lambda_z(n) (mn)^{\frac{1}{2}} \Big(\frac{n}{m}\Big)^{it} \, dt,\] coming from the first terms in the first integral. We get a diagonal term from $m=n$, namely $ T \widehat{\Psi_U}(0)\sum_{n \leqslant X} \Lambda(n) \lambda_z(n) n$, which can be estimated using partial summation and estimate \eqref{eq: hyp 3} below. For the terms $m \neq n$, relabelling $m = n+k$ and rearranging we get \[ \sum\limits_{1 \leqslant \vert k\vert \leqslant X} \sum\limits_{\max(0,- k) < n \leqslant \min(X,X-k)}\Lambda(n) \lambda_z(n+k) n\Big(1 + \frac{k}{n}\Big)^{\frac{1}{2}}\widehat{\Psi_U}(T \log(1 + \frac{k}{n})).\] The fast decay in the Fourier transform $\widehat{\Psi_U}$ can be used to truncate $k$, and an estimate can be derived using partial summation and Lemma \ref{Lemma: correlations on AP} below. 

The details of an extremely general formulation of this approach were worked out by Goldston--Gonek \cite{GG98}: Corollary 1 of \cite{GG98} provides us with a general way of taking bounds from diagonal and off-diagonal correlations of the coefficients of Dirichlet polynomials and using them to generate estimates for expressions such as (\ref{eq: lower bound on I1}). More precisely, this corollary applies to estimating \[ \int\limits_{ - \infty}^{\infty} \Psi_U\Big(\frac{t}{T}\Big)\Big( \sum\limits_{n \leqslant X} a_n n^{-\sigma - it} - \int\limits_{1}^{X} u^{-\sigma - it} \, du\Big)\overline{\Big(\sum\limits_{n \leqslant X} b_n n^{-\sigma - it} - \int\limits_{1}^{X} u^{-\sigma - it} \, du\Big)} \, dt,\] provided the coefficients $a_n,b_n$ satisfy various hypotheses on the summatory functions $\sum_{n \leqslant X} a_n$, $\sum_{n \leqslant X}b_n$, as well as on the correlations $\sum_{n \leqslant X} a_n b_{n+h}$. These hypotheses are given by $(A1)$ to $(A4)$ in \cite{GG98}. 

To formulate these hypotheses for $\Lambda(n)$ and $\lambda_z(n)$, let 
\begin{equation}
\label{eq: def of z}
z = X^\nu
\end{equation} for some parameter $\nu$ to be chosen later. The bound 

\begin{equation}
\label{eq: hyp 1}
\sum\limits_{n \leqslant X} \Lambda(n) = X + O_\varepsilon(X^{1/2 + \varepsilon})
\end{equation}
follows from RH. The bounds
\begin{equation}
\label{eq: hyp 2}
\sum\limits_{ n \leqslant X} \lambda_z(n) = X + O(X^{\nu})
 \end{equation} 
and

\begin{equation}
\label{eq: hyp 3}
\sum\limits_{n \leqslant X} \Lambda(n) \lambda_z(n) = X L(z) + O_{\varepsilon}(X^{\nu + \varepsilon})
\end{equation}
are elementary and may be found in \cite{GGOS00}. Our estimates for $\sum_{n \leqslant X-k} \lambda_z(n) \lambda_z(n+k)$ and $\sum_{n \leqslant X} \lambda_z(n)^2$ were given in Propositions \ref{Proposition lambdaQ correlation} and \ref{Proposition lambdaQ squared on RH}.  Finally, we note a simple consequence of the conjecture $AP(\theta)$. 
\begin{Lemma}
\label{Lemma: correlations on AP}
Let $\theta \in [0,1]$, and assume conjecture $AP(\theta)$. Let $X \geqslant 2$, let $1 \leqslant \vert k\vert\leqslant X$ and let $X_1 = \max(0,-k)$ and $X_2 = \min(X,X-k)$. Let $\nu \in [0, \theta)$ and let $z$ satisfy \eqref{eq: def of z}. Then we have \[\sum\limits_{X_1 < n \leqslant X_2} \lambda_z(n) \Lambda(n+k) = \mathfrak{S}(k)(X-\vert k\vert)  + O_{\varepsilon,\theta}(X^{\max(\frac{1}{2} + \frac{\nu}{2}, 1- \nu) + \varepsilon}),\] 
\end{Lemma}
\begin{proof}
We have
\begin{align*}
\sum\limits_{X_1 < n \leqslant X_2} \lambda_z(n) \Lambda(n+k) &= \sum\limits_{d \leqslant z} \rho_{z,d} \sum\limits_{\substack{X_1 < n \leqslant X_2 \\ d \vert n}} \Lambda(n+k) \\
&= (X- \vert k\vert)\sum\limits_{\substack{d \leqslant z \\ (d,k) = 1}} \frac{ \rho_{z,d}}{\varphi(d)} + O_{\varepsilon, \theta}\Big(X^{1/2 + \varepsilon} \sum\limits_{ d \leqslant z} \frac{ \vert \rho_{z,d}\vert}{d^{1/2}}\Big) \\
& = \mathfrak{S}(k)(X-\vert k\vert) + O\Big( \frac{k \tau(k) X}{\varphi(k) z}\Big) + O_{\varepsilon,\theta}(X^{\frac{1}{2} + \varepsilon} z^{\frac{1}{2}}L(z)) \\
& = \mathfrak{S}(k)(X-\vert k\vert) + O_{\varepsilon,\theta}(X^{\max(\frac{1}{2} + \frac{\nu}{2}, 1- \nu) + \varepsilon})
\end{align*}
\noindent as claimed. To estimate the sum of $\rho_{z,d}\varphi(d)^{-1}$ we have used the calculation from the proof of \cite[Lemma 1]{Go95}, which in turn is taken from \cite[p. 57]{HB85}.  
\end{proof}




\noindent Plugging everything into Corollary 1 of \cite{GG98}, as applied to the integrals in (\ref{eq: lower bound on I1}), we have
\begin{align}
\label{eq: I1 after Corollary 1}
I_{1}(X,T) \geqslant & \widehat{\Psi_U}(0) T\Big(2 \sum\limits_{n \leqslant X} \Lambda(n) \lambda_z(n)n - \sum\limits_{n \leqslant X} \lambda_z^2(n)n \Big) \nonumber\\
& + 4 \pi \Big( \frac{T}{2\pi}\Big)^3 \int\limits_{T/2\pi X}^\infty \Big(\sum\limits_{ h \leqslant 2 \pi X v/T} \mathfrak{S}(h) h^2 \Big) \Re \widehat{\Psi_U}(v) \frac{1}{v^3} \, dv \nonumber\\
& - 4 \pi \Big(\frac{T}{2 \pi}\Big)^3 \int\limits_{XT^\varepsilon/2\pi}^{\infty} \Big( \int\limits_{0}^{2 \pi Xv/T} u^2 \, du \Big) \Re \widehat{\Psi_U}(v) \frac{1}{v^3} \, dv \nonumber\\
& + O_{\varepsilon}(T^{-1} X^{3 + \varepsilon}) + O_{\varepsilon}(X^{5/2 +  \varepsilon}) + O_{\varepsilon}(X^{\varepsilon}) + O_{\varepsilon,\theta}(X^{2 + \varepsilon + \omega}),
\end{align}
where 
\begin{align*}
\omega &= \max\Big(\frac{1}{2}, \nu, \max\Big(1- \nu, \frac{1}{2} + \frac{ \nu}{2}\Big), \max\Big(1-\nu,\min\Big(2 \nu, \frac{47}{74} + \frac{ 53 \nu}{74}\Big)\Big)\Big)\\
& =  \max\Big( \nu, 1- \nu, \frac{1}{2} + \frac{ \nu}{2}, \min\Big(2 \nu, \frac{47}{74} + \frac{ 53 \nu}{74}\Big)\Big).
\end{align*}

The size of the main terms in (\ref{eq: I1 after Corollary 1}) is calculated in \cite[Section 7]{GGOS00}. Here, provided that the error term in Proposition \ref{Proposition lambdaQ squared on RH} is at most $X$, these terms are shown\footnote{We feel we should point out that there is a confusing misprint in this section of \cite{GGOS00}, in the display equation before \cite[equation (7.3)]{GGOS00}, which is incorrect by a factor of $2$. But this factor is then corrected by the end of the calculation in \cite{GGOS00}.} to have a combined size of 
\begin{equation}
\label{eq: main term I1}
\frac{1}{2} X^2T\log (Tz/X) + O(TX^2(\log X)^{2/3} \log\log T).
\end{equation} So the error term in (\ref{eq: I1 after Corollary 1}) will be negligible provided \[ \nu < 1\] and \[\alpha \omega < 1.\]

A similar argument may be used to bound $I_2(X,T)$ from below, deferring to Corollary 2 of \cite{GG98} instead of Corollary 1. We direct the reader to \cite{GGOS00} and \cite{GG98} for the details. The only aspect that it is important for us to note is that the parameter $\eta$ from \cite{GGOS00}, which is taken to be $1/2 - \varepsilon$ in expression (6.3) of that paper, may in fact be taken to be $1$. (Here the parameter $\eta$ comes from hypothesis $(A3)$ of \cite{GG98}, where $X^\eta$ measures the range of uniformity in $k$ for the error terms in the estimates for the correlations $\sum_n \lambda_z(n) \lambda_z(n+k)$ and $\sum_n \lambda_z(n) \Lambda(n+k)$.)

One derives that, uniformly for all $T \ll X \ll T^{2-\varepsilon}$,
\begin{equation}
\label{eq:main term I2}
I_2(X,T) \geqslant \frac{T}{2 X^2} \log(Tz/X) + O\Big(\frac{T}{X^2} (\log X)^{2/3} \log\log T\Big) + E,
\end{equation}
\noindent where $E$ is an error term of size at most  \[O_{\varepsilon}(T^{-1} X^{-1 + \varepsilon}) + O_{\varepsilon}(X^{-2 + \omega + \varepsilon}).\] Thus the main term dominates provided that \[ \alpha \omega < 1.\] 

Therefore from Lemma \ref{Lemma GGOS lemma} we conclude that 
\begin{equation}
\label{eq: lower bound}
F_T(\alpha) \geqslant (1 - \alpha(1 - \nu)) + o_{\theta}(1)
\end{equation}
as $T \rightarrow \infty$, provided that 
\begin{equation}
\label{eq: constraints}
\nu < \theta, \text{ and } \omega < \frac{1}{\alpha}.
\end{equation}

All that remains is to maximise the lower bound (\ref{eq: lower bound}) subject to the constraints (\ref{eq: constraints}). If $\theta \geqslant 27/53$, a straightforward calculation yields 
\begin{equation}
\label{eq: solution}
\nu = \frac{74}{53 \alpha} - \frac{47}{53} - \varepsilon
\end{equation}
as a solution, provided $1 \leqslant \alpha \leqslant 95/94 - \varepsilon$. This should be contrasted with \cite{GGOS00}, in which the authors take $z = T^{1/2 - \varepsilon}$, i.e. $\nu = \frac{1}{2 \alpha} - \varepsilon$, throughout, which is a smaller threshold. However, when $\alpha = 95/94$ we have $\frac{74}{53 \alpha} - \frac{47}{53} = \frac{1}{2 \alpha}$, i.e. the bound \eqref{eq: solution} collapses to the choice in \cite{GGOS00} at the end point of the interval. 

Plugging this value \eqref{eq: solution} of $\nu$ into expression (\ref{eq: lower bound}), we derive Theorem \ref{Theorem Main Theorem}.\qed\\

The proof of Theorem \ref{Theorem lower bound on variance} is a similarly straightforward adaptation of \cite{GY01}, easier even, as one does not need to defer to mean value estimates proved elsewhere. 

In the notation of \cite{GY01}, $R$ is used for our parameter $z$, and we take $q=1$. The key expression from \cite{GY01} is (4.36) of that paper, in which there are several error terms. It may be easily seen that $AP(27/53)$, Proposition \ref{Proposition lambdaQ correlation} and Proposition \ref{Proposition lambdaQ squared on RH} imply that we may remove the \[O(X^{\frac{1}{2}} h^{\frac{3}{2}} R \log^2 X) + O(h^2 R^2)\] error term found there, and replace it with a single \[O_{\varepsilon}(h^2 X^{\frac{47}{74} + \varepsilon} R^{\frac{53}{74} + \varepsilon})\] term, provided $R < X^{\theta}$. Provided $h \leqslant X^{\frac{1}{95} - \varepsilon}$, if one chooses \[ R = \frac{X^{\frac{27}{53} -\varepsilon}}{h^{\frac{74}{53}}}\] then these error terms are of a lower order of magnitude than the main term $xh \log(R/h)$ (which has order of magnitude $hX\log X$). So assuming $AP(27/53)$ and plugging in this value of $R$ into expression (4.36) of \cite{GY01}, one achieves the lower bound in Theorem \ref{Theorem lower bound on variance}. \qed\\

Again, let us note that the choice of $R$ in \cite{GY01}, namely $R = X^{\frac{1}{2} - \varepsilon} h^{-\frac{1}{2}}$, agrees with our choice at the interval end-point $h = X^{\frac{1}{95}}.$\\

\section{Comparison with the alternative hypothesis}
\label{Sec: AH} 
In this short concluding section, we will reflect upon how the various thresholds that we used in studying $F_T(\alpha)$ interact with the so-called Alternative Hypothesis. This hypothesis (as formulated in expression (2.2) of \cite{FGL14}, say) contends that the distribution of the imaginary parts of the zeros of the Riemann zeta function is exceedingly regular. More precisely, AH states that if $\gamma, \gamma^\prime \geqslant T_0$ are such imaginary parts then \[ \widetilde{\gamma} - \widetilde{\gamma^\prime} \in \frac{1}{2} \mathbb{Z},\] where $\widetilde{\gamma}: = \frac{1}{2 \pi}\gamma \log( \gamma/2\pi)$ is the normalised version of $\gamma$. (Certain formulations, see \cite[Conjecture 2.2]{LR19}, allow $ \widetilde{\gamma_j} - \widetilde{\gamma}_{j+1} \in \frac{1}{2} \mathbb{Z} + o_{j \rightarrow \infty}(1)$, where $\gamma_1,\gamma_2,\gamma_3,\dots$ are the positive imaginary parts given in ascending order.) This distribution would contradict Conjecture \ref{Conjecture Montgomery pair corr}. Indeed, under AH, it is immediate that the function $F_T(\alpha)$ is periodic with period $2$. Therefore Theorem \ref{Theorem Montgomery first theorem} implies that under AH the limit $\lim_{T\rightarrow \infty}F_T(\alpha)$ is determined. In particular, expression (\ref{eq: AH FT}) would hold (contradicting Conjecture \ref{Conjecture Montgomery pair corr}). 

However, as is proven in \cite{LR19}, AH is nonetheless consistent with everything that is presently known about the correlations of the zeros (both pair correlations and higher correlations). 

A sufficiently strong disproof of AH would rule out the existence of exceptional zeros for Dirichlet $L$-functions. This fact was known to Montgomery, and is remarked upon at the end of \cite{Mo73}, but the full proof doesn't seem to have been written down until Conrey--Iwaniec \cite{CI02}. Given that the results of our paper are dependent on strong uniformity for the distribution of primes in arithmetic progressions (so in particular they assume the non-existence of exceptional zeros), disproving AH using our methods would say nothing about the exceptional zero problem. However, we nonetheless found it interesting to observe what correlation estimates would be necessary in order to be able to conclude that $\liminf_{T\rightarrow \infty} F_T(\alpha) \geqslant 2 - \alpha$ for $\alpha \in [1,2)$, which is the bound given by (\ref{eq: AH FT}).\\

Let us assume RH, and also write $X_1,X_2$ as in Proposition \ref{Proposition lambdaQ correlation}. Suppose first that $AP(1)$ holds, and also assume that the weight $\lambda_z(n)$ is as well-distributed in arithmetic progressions as we assume $\Lambda(n)$ to be, i.e. assume that if $z = X^{\nu}$ then \[ \sum\limits_{ X_1 < n \leqslant X_2} \lambda_z(n) \lambda_z(n+k) = \mathfrak{S}(k)(X-\vert k\vert) + O_{\varepsilon}\Big(\frac{ \tau(k) k X}{\varphi(k) z}\Big) + O_{\varepsilon}(X^{\frac{1}{2} + \frac{\nu}{2} + \varepsilon}),\] for all $\nu < 1$, for all $1 \leqslant \vert k\vert \leqslant X$. Then the G--G--\"{O}--S method shows that, for each fixed $\alpha \in [1,3/2)$, \[\liminf \limits_{T \rightarrow \infty } F_T(\alpha) \geqslant 3 - 2 \alpha.\] 

This bound still vanishes at $\alpha = 3/2$. In order to approach the AH bound, and move beyond $3/2$, one would need to assume some extra cancellation. Of course, as we remarked in the introduction, if for some $\eta \geqslant 1/2$ one goes as far as assuming asymptotics of the form \begin{equation}
\label{eq: twin prime necessity}
\sum\limits_{X_1 \leqslant n \leqslant X_2} \Lambda(n)\Lambda(n+k) = \mathfrak{S}(k)(X - k) + O_{\varepsilon}(X^{\eta + \varepsilon}),
\end{equation}
uniform in $1 \leqslant \vert k \vert \leqslant X^{1- \eta}$, then one would resolve Conjecture \ref{Conjecture Montgomery pair corr} in the range $\alpha < \eta^{-1}$. (This follows from the mean value theorems of Goldston-Gonek \cite{GG98}, for instance, or from Goldston-Montgomery \cite{GM87}.) However, suppose one only assumed results on the cross-correlations of $\lambda_z$ and $\Lambda$, and on the correlations of $\lambda_z$ with itself. Suppose we had the following:
\begin{equation}
\label{eq: assumption to get AH bound}
\sum\limits_{ X_1 < n \leqslant X_2} \lambda_z(n) \Lambda(n+k) = \mathfrak{S}(k) (X - \vert k\vert )  + O_{\varepsilon}(z X^{\varepsilon})
\end{equation} 
and
\begin{equation}
\label{eq: assumption to get AH bound 2}
\sum\limits_{ X_1 < n \leqslant X_2} \lambda_z(n) \lambda_z(n+k) = \mathfrak{S}(k) (X - \vert k\vert ) + O\Big( \frac{\tau(k) k X}{\varphi(k) z} \Big) + O_{\varepsilon}(z X^{\varepsilon})
\end{equation} 
for all $z > X^{1/2}$, i.e. an error term that matches, up to an $X^{\varepsilon}$ factor, the known error term for the simpler sum $\sum_{n \leqslant X} \lambda_z(n)$ from \eqref{eq: hyp 2}. Then the method of Section \ref{Sec:method of GGOS} shows that we may take $z = T^{1 - \varepsilon}$, i.e. $\nu = \frac{1}{\alpha} - \varepsilon$, and establish that \[ F_T(\alpha) \geqslant 1 - \alpha(1 - \nu) - o_{\varepsilon}(1) \geqslant 2 - \alpha - \varepsilon - o_{\varepsilon}(1)\] as $T \rightarrow \infty$. And since $\varepsilon$ is arbitrary, this establishes $\liminf_{T\rightarrow \infty} F_T(\alpha) \geqslant 2 - \alpha$, i.e. a lower bound that matches AH.

Thus it seems reasonable to describe Theorem \ref{Theorem GGOS} of G--G--\"{O}--S as `half' the AH bound. Indeed, G--G--\"{O}--S work with a sieving level $z = T^{\frac{1}{2} - \varepsilon}$; if one could double the exponent and take $z = T^{1- \varepsilon}$ then one could achieve the AH bound. \\

The bound \eqref{eq: assumption to get AH bound 2} would follow from square-root cancellation in the Type II sums \eqref{eq: before Bettin Chandee}. However, to prove \eqref{eq: assumption to get AH bound} -- even assuming $AP(1)$ -- one would nonetheless need to detect extra cancellations in the sum \[ \sum\limits_{ d \leqslant z} \rho_{z,d} \Big(\sum\limits_{ \substack{ X_1 < n \leqslant X_2 \\ d\vert n}} \Lambda(n+k) - 1_{(k,d) = 1} \frac{(X - \vert k\vert)}{\varphi(d)}\Big)\] coming from the oscillations in $\rho_{z,d}$, above and beyond the cancellation from Conjecture \ref{Conjecture Montgomery APs}. For $z$ just a little larger than $X^{1/2}$, there is of course a history of using `well-factorability' of certain linear sieve weights $\rho_{z,d}$ to prove better distribution in arithmetic progressions than one could otherwise prove (the celebrated work of Bombieri--Friedlander--Iwaniec \cite{BFI86}, say, and more modern work by Maynard \cite{Ma20}). To make progress along these lines one would need to replace the Selberg weight $\lambda_{z,d}$ with a different weight. Yet, seeing as we are not sieving but rather taking `major arc approximations' to $\Lambda(n)$, there are fewer choices of weights at our disposal. In \cite[p.367]{Go95} Goldston remarks that certain other difficulties arise when using a weight such as $\sum_{d \vert n: d \leqslant z} \mu(d) \log(z/d)$ (concerning the approximation of the singular series for the correlations).

It remains a subject for future research to establish whether there is a range of $\alpha$ for which the weight $\sum_{d \vert n: d \leqslant z} \mu(d) \log(z/d)$, in combination with B--F--I machinery, yields a direct improvement to Theorem \ref{Theorem GGOS} -- only assuming GRH.

\bibliographystyle{plain}
\bibliography{paircorr}
\end{document}